\documentclass[11pt,letterpaper]{amsart} %

\usepackage[letterpaper,margin=1in]{geometry}
\usepackage{amsmath, amssymb, amsthm}
\usepackage[colorlinks=true, citecolor=blue, urlcolor=blue, linkcolor=blue]{hyperref}

\newtheorem{theorem}{Theorem}[section]

\newtheorem{corollary}[theorem]{Corollary}
\theoremstyle{definition}
\newtheorem{definition}[theorem]{Definition}
\newtheorem{remark}[theorem]{Remark}

\newtheorem{apptheorem}{Theorem}[section]
\newtheorem{applemma}[apptheorem]{Lemma}

\begin{document}

\title{The Equations of General Hassett maximal Cubic Fourfolds} 

\author{Elad Gal}
\address{Faculty of Mathematics \\ Technion, Israel Institute of Technology} 
\email{galelad@campus.technion.ac.il}

\author{Howard Nuer}
\address{Faculty of Mathematics \\ Technion, Israel Institute of Technology}
\email{hnuer@technion.ac.il} 

\begin{abstract}
In this note, we discuss Hassett maximal cubic fourfolds and construct an explicit irreducible component of maximal dimension sixteen of the locus $\mathcal{Z}$ of Hassett maximal cubic fourfolds.
We utilize algebraic and arithmetic methods to analyze the associated lattice of these fourfolds. %
By studying general integral quadratic forms and proving the ADC property for a specific ternary form, we demonstrate that the primitive image of our lattice spans the entire Hassett subset, confirming the Hassett maximality of the cubic fourfolds we describe.
\end{abstract}

\maketitle

\section{Introduction}

Cubic fourfolds occupy an important place in modern algebraic geometry, largely due to their rich Hodge theory and the elusive nature of their rationality problem. %
For a cubic fourfold $X \subset \mathbb{P}^5$, the algebraic cohomology lattice $A(X) := H^4(X, \mathbb{Z}) \cap H^{2,2}(X)$ is positive-definite and always contains the square of the hyperplane class, $h_X^2$. %

A cubic fourfold is termed \emph{special} if the rank of its associated lattice $A(X)$ is at least 2. %
Hassett \cite{Hassett} showed that the locus of special cubic fourfolds is a union of irreducible divisors $\mathcal{C}_d$, where $d$ is the discriminant, and these divisors are nonempty if and only if $d$ is in the \emph{Hassett subset} $\mathcal{H} := \{n \in \mathbb{N}_{\ge 8} : n \equiv 0, 2 \pmod 6\}$. %

Generalizing this, Yang and Yu \cite{YangYu} defined an \emph{$M$-polarizable cubic fourfold} as one admitting a primitive embedding of the positive-definite lattice $M$ into $A(X)$ such that $h_X^2 \in M$, and they introduced the subvariety $\mathcal{C}_M$ of the moduli space $\mathcal{C}$ of cubic fourfolds parametrizing those cubic fourfolds that are $M$-polarizable. %
The intersection of all Hassett divisors, denoted $\mathcal{Z} := \bigcap_{d \in \mathcal{H}} \mathcal{C}_d$, represents fourfolds with maximally rich geometric structures and were coined \emph{Hassett maximal} in \cite{Marquand}. %
In our previous work \cite{GalNuer}, we established that the dimension of $\mathcal{Z}$ is 16 by first showing that any irreducible component must be of the form $\mathcal{C}_M$ for some positive definite lattice $M$ of minimal rank among those primitively representing $\mathcal{H}$, and then showing that this minimal rank is four. %

However, the explicit relationship between the lattice $M$ and the geometry of those cubic fourfolds $[X] \in \mathcal{C}_M$ remained opaque. %
The purpose of this note is to construct an explicit family of cubic polynomials whose vanishing sets are Hassett maximal cubic fourfolds, providing a concrete geometric description of the cubic fourfolds in a sixteen dimensional irreducible component $\mathcal{C}_M$. %
Specifically, we prove the following theorem:
\begin{theorem}\label{thm:Main theorem}
    There exists an irreducible component of $\mathcal{Z}$ of dimension $16$ whose general member is isomorphic to the vanishing locus of some homogeneous cubic polynomial
    $$F \in ( x,y,z ) \cap ( x,y,u ) \cap ( x,z,v ) \cap ( v-by,u-az,w )$$
    for $a,b\in \mathbb{C}$.
\end{theorem}

\subsection*{Acknowledgments}
HN would like to thank Lisa Marquand for numerous discussions about cubic fourfolds and for her interest in this work.
HN was partially supported by ISF Grant 3175/25.

\section{A New Family of Hassett Maximal Cubic Fourfolds}

We begin with a foundational result by Voisin \cite{Voisin} (cf. Yang and Yu \cite[Theorem 6.2]{YangYu}) regarding the intersection of planes in a cubic fourfold, which translates geometric intersections into lattice pairings: %

\begin{theorem}[Voisin]
Let $X \subset \mathbb{P}^5$ be a cubic fourfold and $P_1, P_2 \subset X$ be two planes. Then:
\begin{itemize}
    \item $[P_1] \cdot [P_2] = 0 \iff P_1 \cap P_2 = \emptyset$ %
    \item $[P_1] \cdot [P_2] = 1 \iff P_1 \cap P_2 \text{ is a point}$ %
    \item $[P_1] \cdot [P_2] = -1 \iff P_1 \cap P_2 \text{ is a line}$ %
\end{itemize}
In the third case, there exists a third plane $P_{\text{residual}} \subset X$ such that $[P_{\text{residual}}] \cdot [P_1] = [P_{\text{residual}}] \cdot [P_2] = -1$, and $h_X^2 = [P_1] + [P_2] + [P_{\text{residual}}]$. %
\end{theorem}

Inspired by this, we define a specialized family of cubic fourfolds: 

\begin{definition}
Let $\mathcal{N}$ denote the set of all cubic fourfolds that contain four planes $P_1, P_2, P_3, P_4$ exhibiting the following intersection profile: %
\begin{itemize}
    \item $P_1 \cap P_2$ is a line. 
    \item $P_1 \cap P_3$ is a line. 
    \item $P_1 \cap P_4 = \emptyset$. 
    \item $P_2 \cap P_3$ is a point. 
\end{itemize}
For any $X \in \mathcal{N}$, the lattice $A(X)$ contains a primitive rank-5 sublattice $M$ generated by $\mathfrak{o} = h_X^2$ and the plane classes. %
Based on Theorem 2.1, the Gram matrix of $M$ takes the form: %
\begin{equation}\label{eqn:matrix for M}
M = \begin{pmatrix} 
3 & 1 & 1 & 1 & 1 \\ 
1 & 3 & -1 & -1 & 0 \\ 
1 & -1 & 3 & 1 & \alpha \\ 
1 & -1 & 1 & 3 & \beta \\ 
1 & 0 & \alpha & \beta & 3 
\end{pmatrix}
\end{equation}
where the parameters $\alpha, \beta \in \{0,1\}$ correspond to the possible orbits of plane arrangements, as detailed in Theorem 2.4.
\end{definition}

\begin{theorem}
Let $F$ be a smooth homogeneous cubic polynomial defining a fourfold $V_+(F) \in \mathcal{N}$ in $\mathbb{P}^5 = \mathbb{P}\mathbb{C}^6$. Then there exists a basis of linear forms $\{x,y,z,u,v,w\}$ for the dual space $(\mathbb{C}^6)^*$ and constants $a,b \in \mathbb{C}$ such that:
$$F \in ( x,y,z ) \cap ( x,y,u ) \cap ( x,z,v ) \cap ( v-by,u-az,w )$$
\end{theorem}

\begin{proof}
Let $X = V_+(F) \in \mathcal{N}$, and let $P_1,P_2,P_3,P_4 \subset \mathbb{P}^5$ be the four planes exhibiting the required incidence profile. Each plane $P_i$ is the projectivization of a 3-dimensional vector subspace $V_i \subset \mathbb{C}^6$. The defining ideal $I(P_i)$ is generated by the 3-dimensional subspace of linear forms in $(\mathbb{C}^6)^*$ that vanish on $V_i$.

Since $P_1$ is a plane, we can choose three linearly independent linear forms $x,y,z \in (\mathbb{C}^6)^*$ such that $I(P_1) = ( x,y,z )$. 

Because the intersection $P_1 \cap P_2$ is a line, the sum $I(P_1) + I(P_2)$ must be generated by exactly 4 independent linear forms. Therefore, $I(P_2)$ shares a 2-dimensional subspace of linear forms with $I(P_1)$. We can choose the generators $x,y$ to span this intersection and extend with a new independent linear form $u$ such that $I(P_2) = ( x,y,u )$.

Similarly, $P_1 \cap P_3$ is a line, so $I(P_3)$ also shares a 2-dimensional subspace of linear forms with $I(P_1)$. For the intersection $P_2 \cap P_3$ to be exactly a point, the sum $I(P_2) + I(P_3)$ must be generated by 5 independent linear forms. This forces the shared forms between $I(P_1)$ and $I(P_3)$ to be spanned by $x$ and a form independent of $y$, which we take to be $z$. Extending with a fifth independent linear form $v$, we obtain $I(P_3) = ( x,z,v )$. Notice that $I(P_2) + I(P_3) = ( x,y,z,u,v )$, which correctly cuts out a point.

Finally, because $P_1 \cap P_4 = \emptyset$, the sum $I(P_1) + I(P_4)$ must be generated by the entire 6-dimensional dual space $(\mathbb{C}^6)^*$. This allows us to choose the remaining basis element $w$ such that $\{x,y,z,u,v,w\}$ forms a complete basis for $(\mathbb{C}^6)^*$. By applying suitable linear changes of basis to the generators of $I(P_4)$ while preserving our established basis elements, we can put the ideal of $P_4$ into the canonical form:
$$I(P_4) = ( v-by,u-az,w )$$
for some constants $a,b \in \mathbb{C}$. Because the fourfold $X = V_+(F)$ contains all four planes $P_i$, the polynomial $F$ must vanish on each $P_i$, placing it precisely in the intersection of their respective ideals:
$$F \in ( x,y,z ) \cap ( x,y,u ) \cap ( x,z,v ) \cap ( v-by,u-az,w ).$$
This completes the proof.
\end{proof}
Now we show that this construction dominates the irreducible sixteen dimensional subvariety $\mathcal{C}_M$.
\begin{theorem}
$\dim(\mathcal{N}) = 16$, implying that the generic element of $\mathcal{C}_M$ is contained in $\mathcal{N}$. %
\end{theorem}

\begin{proof}
Consider the parameter space of the plane arrangements $\mathcal{P} := \{(P_1, P_2, P_3, P_4) \mid P_i \cap P_j \text{ is as required}\} \subset \text{Gr}(3,6)^4$. %
The group $\text{PGL}(6, \mathbb{C})$ acts on $\mathcal{P}$ diagonally. %
Up to this action, the planes can be classified by $\alpha, \beta \in \{0, 1\}$. %
The dimension of the $\text{PGL}(6, \mathbb{C})$-orbit $O_{\alpha,\beta}$ is $28 - \delta_{\alpha 0} - \delta_{\beta 0}$. %

Analyzing the incidence correspondence $S := \{((P_1,\dots,P_4), X) \mid \bigcup P_i \subset X\} \subset \mathcal{P} \times |\mathcal{O}_{\mathbb{P}^5}(3)|$, the fiber dimension over a point in $O_{\alpha,\beta}$ is $23 + \delta_{\alpha 0} + \delta_{\beta 0}$. %
Therefore, the total dimension of $p_1^{-1}(O_{\alpha,\beta})$ is $51$. %
Since a smooth cubic fourfold contains at most 405 planes \cite{DIO}, the projection to the moduli space of cubics is quasi-finite. %
Factoring out the 35-dimensional $\text{PGL}(6, \mathbb{C})$ action yields $51 - 35 = 16$ for the image $\mathcal{N}$ in the moduli space $\mathcal{C}$. %
\end{proof}

\begin{remark}
While technically we should have written $\mathcal{N}_{\alpha,\beta}$ and $M_{\alpha,\beta}$ as the parameter space $\mathcal{P}$ splits into four distinct $\text{PGL}(6, \mathbb{C})$-orbits corresponding to $(\alpha, \beta) \in \{(0,0), (1,0), (0,1), (1,1)\}$, these four configurations define the exact same abstract lattice $M$ up to isometry, and thus collapse into a single irreducible component in the moduli space $\mathcal{C}$.

Algebraically, if we denote the plane classes as $a = [P_1]$, $b = [P_2]$, $c = [P_3]$, and $d = [P_4]$, the $(0,0)$ configuration corresponds to $d\cdot b = 0$ and $d\cdot c = 0$. Using the relation $h_X^2 = [P_1] + [P_2] + [P_{\text{residual}}]$ from Theorem 2.1, we can define new classes $b' := \mathfrak{o} - a - b$ and $c' := \mathfrak{o} - a - c$. Substituting these transforms the intersection numbers from $0$ to $1$, demonstrating that all four intersection matrices are integer basis changes of the same lattice $M$.

Geometrically, Voisin's theorem implies that the hyperplane $H \cong \mathbb{P}^4$ spanned by the intersecting planes $P_1$ and $P_2$ cuts out a cubic threefold in $X$ that splits into $P_1 \cup P_2 \cup P_{\text{residual}}$. If $P_4$ is disjoint from $P_2$, it must intersect $P_{\text{residual}}$ in a point. Thus, the four orbits merely represent choices of different planes within the exact same family of cubic fourfolds.
\end{remark}

\section{ADC Forms and Hassett-Maximality}
Now that we have established that $\mathcal{N}$ fills an irreducible subvariety $\mathcal{C}_M$ of dimension $16$, it remains to prove that $\mathcal{C}_M$ is an irreducible component of the intersection of all Hassett divisors $\mathcal{Z}$.  For this we must analyze the lattice $L(M)$, see \cite[Definition 3.2]{GalNuer}. %
By utilizing the distinguished element $\mathfrak{o}$, we obtain a rank-4 lattice with the associated positive-definite quadratic form: 
$$F(x,y,z,u) = 4(2x^2+2y^2+2z^2-2xy-2xz+yz)+2u(4u-x-y-z)$$ %
which can be diagonalized as: 
$$F(x,y,z,u) = \frac{1}{8}((8x-4y-4z-u)^2+3(4y-u)^2+3(4z-u)^2+57u^2)$$ %

We must prove that the primitive image $PIm(F)$ is exactly $\mathcal{H}$. %
To do so, we deploy the foundational theory of ADC forms. 

\begin{definition}
An integral quadratic form is called an \textbf{ADC form} if any integer representable by a rational vector is also representable by an integer vector. %
\end{definition}

\begin{theorem}\label{Thm:Q and G are ADC}
The ternary quadratic form $Q_3(x,y,z) = x^2+y^2+3z^2$ and the form $G(x,y,z) = x^2+3y^2+3z^2$ are ADC forms. %
\end{theorem}

\begin{proof}
The fact that $Q_3(x,y,z) = x^2+y^2+3z^2$ is an ADC form is a classically established property of regular ternary forms (for example see \cite[Table 2]{ClarkJagy}); a self-contained, geometric proof extending Davenport-Cassels descent via torus partitioning is provided for completeness in Appendix \ref{app:adc}. 

To see that $G$ inherits the ADC property directly from $Q_3$, observe the algebraic identity:
$$Q_3(3y,3z,x) = (3y)^2 + (3z)^2 + 3x^2 = 3(x^2+3y^2+3z^2) = 3G(x,y,z)$$
Suppose $n \in \mathbb{Z}$ is represented rationally by $G$, meaning $3n$ is represented rationally by $Q_3$. Because $Q_3$ is an ADC form, there exist integers $A,B,C \in \mathbb{Z}$ such that $Q_3(A,B,C) = 3n$, or $A^2+B^2+3C^2 = 3n$. Reducing this equation modulo 3 yields $A^2+B^2 \equiv 0 \pmod 3$. Since $-1$ is not a quadratic residue modulo 3, this requires $A \equiv 0 \pmod 3$ and $B \equiv 0 \pmod 3$. Writing $A=3Y$ and $B=3Z$ for integers $Y,Z$, and letting $X=C$, we obtain $3n = Q_3(3Y,3Z,X) = 3G(X,Y,Z)$. Dividing by 3 gives $n = G(X,Y,Z)$ over the integers, confirming that $G$ is an ADC form. %
\end{proof}
\begin{theorem}
$PIm(F) = \mathcal{H}$. %
\end{theorem}

\begin{proof}
It is immediate that $F(\mathbb{Z}^{\oplus 4}) \subset \mathcal{H}$ since the associated lattice $M$ satisfies the conditions for lattice polarizability from \cite[Theorem 5.1]{YangYu}. %
To establish that every element of $\mathcal{H}$ is primitively represented, we first explicitly construct primitive vectors for the lowest exceptional values:
\begin{itemize}
    \item $v_{24} = (1,0,-1,0)$ yields $F(v_{24}) = 24$, %
    \item $v_{42} = (0,1,-2,1)$ yields $F(v_{42}) = 42$, %
    \item $v_{60} = (1,3,-1,0)$ yields $F(v_{60}) = 60$. %
\end{itemize}
For the remaining values $n \in \mathcal{H} \setminus \{24,42,60\}$, we evaluate $F$ on vectors with a half-integer first coordinate. %
Specifically, observing that $F\left(\frac{x}{2},y,z,u\right) \equiv xu \pmod 2$, %
it follows that for any vector $v = \left(\frac{x}{2},y,z,u\right) \in \frac{\mathbb{Z}}{2} \times \mathbb{Z}^{\oplus 2} \times (2\mathbb{Z}+1)$, the image $F(v)$ is an even integer if and only if $x$ is an even integer, which implies $v \in \mathbb{Z}^{\oplus 4}$. %
Furthermore, if $u \in \{1,-3\}$, any common divisor of the coordinates of $v$ must divide $u$. %
When $u=1$, the vector is automatically primitive; when $u=-3$, any common divisor must divide $3$, but since we will ensure $F(v) = n \not\equiv 0 \pmod 9$, the vector cannot be divisible by $3$ and is thus primitive. %

We define the linear transformation $T : \mathbb{Z}^{\oplus 4} \to \mathbb{Z}^{\oplus 3}$ by:
$$T(x,y,z,u) := 4(x-y-z, y, z) - u(1,1,1)$$ %
which satisfies the fundamental algebraic identity:
$$F\left(\frac{x}{2},y,z,u\right) = \frac{G(T(x,y,z,u)) + 57u^2}{8}$$ %
Because the linear map $(x,y,z) \mapsto (x-y-z,y,z)$ is an integer isomorphism, we have $T(\mathbb{Z}^{\oplus 3} \times \{u\}) = (4\mathbb{Z}+u)^{\oplus 3}$. %
Consequently, representing $n \in \mathcal{H}$ amounts to finding an integer solution to:
$$8n - 57u^2 \in G((4\mathbb{Z}+u)^{\oplus 3})$$ %
We partition the target values based on their residues modulo $18$ and choose $u \in \{1,-3\}$ accordingly: %
$$\mathcal{K} := \{8n-57 : n \in \mathcal{H}, n \not\equiv 6 \pmod{18}\} \cup \{8n-513 : n \in \mathcal{H}, n > 60, n \equiv 6 \pmod{18}\}$$ %
Every integer $k \in \mathcal{K}$ satisfies four core arithmetic properties: %
\begin{enumerate}
    \item $k > 0$, since $n \ge 8$ for $u=1$ and $n > 60$ for $u=-3$. %
    \item $k \equiv 7 \pmod 8$, because $8n - 57 \equiv -57 \equiv 7 \pmod 8$ and $8n - 513 \equiv -513 \equiv 7 \pmod 8$. %
    \item $k \equiv 0,1 \pmod 3$, since $57, 513 \equiv 0 \pmod 3$ and $n \equiv 0,2 \pmod 3$. %
    \item $k \not\equiv 0 \pmod 9$. For $u=1$, $8n-57 \equiv -n-3 \pmod 9$, which is divisible by $9$ if and only if $n \equiv 6 \pmod{18}$, a case explicitly excluded; for $u=-3$, $8n-513 \equiv 8(6) \equiv 3 \not\equiv 0 \pmod 9$. %
\end{enumerate}

To prove $\mathcal{K} \subset G((4\mathbb{Z}+u)^{\oplus 3})$, we first establish the image identity $G((4\mathbb{Z}+1)^{\oplus 3}) = G((2\mathbb{Z}+1)^{\oplus 3}) = G(\mathbb{Z}^{\oplus 3}) \cap (8\mathbb{N}+7)$. %
The first equality holds because $G(a,b,c) = G(\pm a, \pm b, \pm c)$, allowing arbitrary sign choices to ensure odd integers lie in $4\mathbb{Z}+1$. %
For the second equality, any $x,y,z \in 2\mathbb{Z}+1$ clearly yields $G(x,y,z) = x^2+3y^2+3z^2 \equiv 1+3+3 = 7 \pmod 8$. %
Conversely, if $k = x^2+3y^2+3z^2 \equiv 7 \pmod 8$, reducing modulo $4$ requires $x^2-y^2-z^2 \equiv 3 \pmod 4$, forcing at least one of $y,z$, say $z$, to be odd. %
Then $z^2 \equiv 1 \pmod 8$, so $m := x^2+3y^2 = k - 3z^2 \equiv 7 - 3 = 4 \pmod 8$. %
Writing $x=2X$ and $y=2Y$, we have $X^2+3Y^2 \equiv 1 \pmod 2$, meaning $X$ and $Y$ have different parities. %
The transformation $x' = X+3Y$ and $y' = X-Y$ then produces odd integers $x',y'$ satisfying $(x')^2+3(y')^2 = 4(X^2+3Y^2) = m$, confirming the existence of an odd representation. %

Since $G$ is an ADC form by Theorem~\ref{Thm:Q and G are ADC}, $G(\mathbb{Z}^{\oplus 3}) = G(\mathbb{Q}^{\oplus 3})$. %
By the Hasse-Minkowski theorem, $k \in \mathcal{K}$ lies in $G(\mathbb{Q}^{\oplus 3})$ if and only if it is represented over $\mathbb{R}$ and over $\mathbb{Q}_p$ for all primes $p$. %
Representation over $\mathbb{R}$ is immediate since $k > 0$ and $G(\sqrt{k},0,0) = k$. %
Over $\mathbb{Q}_2$, $k \equiv 7 \pmod 8$ is directly represented over $\mathbb{Z}_2$ via the base solution $1^2+3(1^2)+3(1^2) = 7$. %
Over $\mathbb{Q}_3$, the conditions $k \equiv 0,1 \pmod 3$ and $k \not\equiv 0 \pmod 9$ ensure that $k$ avoids local obstructions; by Hensel's lemma (cf. Serre \cite[Chapter IV, Theorem 6]{Serre}), non-zero derivatives modulo $3$ guarantee that solutions lift to $\mathbb{Z}_3$. %
Finally, for any odd prime $p \neq 3$, we utilize the rational equivalence $k \in G(\mathbb{Q}_p^{\oplus 3}) \iff 9p^2k \in G(\mathbb{Q}_p^{\oplus 3})$. %
Setting $x=3$, this reduces to solving $3y^2+3z^2 = 9p^2k - 9$, or equivalently representing $3p^2k - 3$ as a sum of two squares $y^2+z^2$ over $\mathbb{Q}_p$. %
Because $3p^2k-3\equiv-3\not\equiv0 \pmod{p}$, the value is a $p$-adic unit. Representing a unit as a sum of two squares over $\mathbb{Z}_p$ is equivalent to the isotropy of the ternary form $\langle 1, 1, -(3p^2k-3) \rangle$; since all coefficients are units and $p \neq 2$, this form is isotropic over the residue field $\mathbb{F}_p$ and lifts to a non-trivial zero over $\mathbb{Z}_p$ via Hensel's lemma \cite{parimala2013}.
Thus, $k$ is globally represented, completing the proof that $PIm(F) = \mathcal{H}$. %
\end{proof}

\begin{corollary}
$\mathcal{N} \subset \mathcal{C}_M \subset \mathcal{Z}$, and $\dim(\mathcal{Z}) = 16$. %
Every $[X] \in \mathcal{C}_M$ is Hassett-maximal. %
\end{corollary}

Because the primitive image of $L(M)$ exhausts the Hassett subset $\mathcal{H}$, the 16-dimensional family $\mathcal{N}$ densely parametrizes a single irreducible component of the intersection of all Hassett divisors, which completes the proof of Theorem~\ref{thm:Main theorem}. %

\appendix

\section{Geometric Descent for Non-Euclidean Forms}
\label{app:adc}

While Davenport and Cassels classically proved the ADC property for integer-matrix Euclidean forms \cite[Theorem 8]{Clark}, we provide here a generalized geometric shrinking argument that extends descent to non-Euclidean spaces via torus partitioning, confirming independently that $Q_3$ and $G$ are ADC forms.
The result may be of interest in its own right.

\begin{applemma}
\label{lem:torus}
Let $Q:\mathbb{Z}^n \to \mathbb{Z}$ be an integral quadratic form, and define $M := \sup\{Q(v) : v \in [-1,1]^n\}$. For each integer vector $v$ and prime $p > \lceil \sqrt{M} \rceil^n$ such that $Q(v/p) = m \in \mathbb{Z}$, there exists $v' \in \mathbb{Z}^n$ and integers $0 < i, t < p$ such that $Q(v/p) = Q(v'/(it))$.
\end{applemma}

\begin{proof}
If $p \mid v$, we simply choose $v' = v/p$ and $i=t=1$. Otherwise, for each $0 < i < p$, the vector $iv/p$ is not integral. Consider the projection $\pi: \mathbb{R}^n \to \mathbb{T}^n = \mathbb{R}^n/\mathbb{Z}^n$. We examine the $p$ points $\{\pi(iv/p)\}_{i=0}^{p-1}$. We can separate $\mathbb{T}^n$ into $\lceil\sqrt{M}\rceil^n$ distinct $n$-dimensional cubes of side length $1/\lceil\sqrt{M}\rceil$. Since $p > \lceil\sqrt{M}\rceil^n$, the pigeonhole principle implies that two points must lie in the same cube. Their difference, which is the image via $\pi$ of a multiple of $v/p$, must therefore lie in the cube $\frac{1}{\lceil\sqrt{M}\rceil}[-1,1]^n$. 

Thus, there exists $1 \le i < p$ such that:
$$Q\left(\pi\left(i\frac{v}{p}\right)\right) = \frac{1}{\lceil\sqrt{M}\rceil^2} Q\left(\pi\left(i\frac{v}{p}\right)\lceil\sqrt{M}\rceil\right) \le \frac{M}{\lceil\sqrt{M}\rceil^2} < 1$$
Let $x_{\text{new}}$ be this representative of $i v / p$, and let $z \in \mathbb{Z}^n$ be its nearest integer shift such that $|Q(x_{\text{new}} - z)| < 1$. Crucially, while classical Davenport-Cassels descent relies on Euclidean coordinate rounding to establish this bound, our torus pigeonhole argument bypasses the non-Euclidean nature of $Q$ by scaling the vector until the required strict analytic bound is forced. Because the secant line descent formula itself is purely algebraic and independent of metric properties, feeding it the bounded difference vector $x_{\text{new}} - z$ successfully constructs an associated rational vector $v'/t$ with the same norm $i^2m$ whose denominator satisfies $t < p$. Dividing by $i$ yields the desired equivalence.
\end{proof}

\begin{apptheorem}
\label{thm:general_adc}
Let $Q$ and $M$ be defined as above, and let $k \in \mathbb{N} \cap Q(\mathbb{Q}^n)$. Then there exists a rational solution $Q\left(\frac{x}{\prod p_j^{r_j}}\right) = k$ where $x \in \mathbb{Z}^n$ and every prime factor $p_j$ in the denominator satisfies $p_j \le \lceil \sqrt{M} \rceil^n$.
\end{apptheorem}

\begin{proof}
Let $u$ be a rational vector with $Q(u) = k$. We write $u = v/(p^r d)$ where $p$ is the maximal prime factor of the denominator, $r$ is its multiplicity, and $d$ represents the remaining factors. If $p > \lceil \sqrt{M} \rceil^n$, we multiply $u$ by $dp^{r-1}$ to isolate the denominator $p$. Applying Lemma \ref{lem:torus} yields a new vector with denominator strictly less than $p$. Applying this iteratively shrinks the maximal prime component until all primes $p_j$ are bounded by $\lceil \sqrt{M} \rceil^n$.
\end{proof}

\begin{apptheorem}
The ternary quadratic form $Q_3(x,y,z) = x^2+y^2+3z^2$ and the form $G(x,y,z) = x^2+3y^2+3z^2$ are ADC forms.
\end{apptheorem}

\begin{proof}
Evaluating the supremum for $Q_3$ yields $M = 5$. Applying the general torus partition bound from Theorem \ref{thm:general_adc} guarantees denominator reduction for primes $p > \lceil \sqrt{5} \rceil^3 = 3^3 = 27$. However, because $Q_3$ is a specific diagonal form, we can bypass coarse box-partitioning entirely. Explicit fractional distance optimization utilizing the geometry of numbers reveals the much tighter prime bound $p \le \sqrt{2(3)} = \sqrt{6} \approx 2.45 < 3$. 

Thus, any rational solution implies an integer solution up to factors of $2$. Therefore $Q_3(x/2^r) = n \implies Q_3(x) = 4^rn$. Modulo 4 considerations confirm that the integer image of $Q_3$ is closed under division by 4, establishing it independently as an ADC form. The deduction that $G$ inherits this property follows precisely as detailed in Theorem~\ref{Thm:Q and G are ADC}.
\end{proof}

\end{document}